\documentclass[a4,11pt]{article}
\usepackage[subrefformat=parens]{subcaption}
\usepackage[dvips]{graphicx}
\usepackage{amssymb}
\usepackage{amstext}
\usepackage{amsmath}
\usepackage{amscd}
\usepackage{amsthm}
\usepackage{amsfonts}
\usepackage{enumerate}
\usepackage{graphicx}
\usepackage{latexsym}
\usepackage{mathrsfs}
\usepackage{mathtools}
\usepackage{cases}
\usepackage[svgnames]{xcolor}
\usepackage{verbatim}
\usepackage{bm}
\usepackage{tikz}
\newtheorem{thm}{Theorem}[section]

\newtheorem{lem}[thm]{Lemma}

\theoremstyle{definition}

\newtheorem{dfn}[thm]{Definition}

\newtheorem*{claim*}{Claim}
\theoremstyle{remark}

\numberwithin{equation}{section}

\title{\vspace{-3cm}\textbf{The monotonicity method for the inverse crack scattering problem}}

\author{Tomohiro DAIMON\footnote{AISIN SOFTWARE Co., Ltd., Japan.}, Takashi FURUYA\footnote{Graduate School of Mathematics, Nagoya University, Japan.}, Ryuji SAIIN\footnotemark[1] }

\date{}
\begin{document}
\maketitle
\begin{abstract}
The monotonicity method for the inverse acoustic scattering problem is to understand the inclusion relation between an unknown object and artificial one by comparing the far field operator with artificial operator. This paper introduces the development of this method to the inverse crack scattering problem. Our aim is to give the following two indicators: One (Theorem 1.1) is to determine whether an artificial small arc is contained in the unknown arc. The other one (Theorem 1.2) is whether an artificial large domain contain the unknown one. Finally, numerical examples based on Theorem 1.1 are given.
\end{abstract}
\date{{\bf Key words}. inverse scattering, monotonicity method, crack detection, far field operator, helmholz equation}
%%%%%%%%%%%%%%%%%%%%%%%%%%%%%%%%%%%%%%%%%%%%%%%%%%%%%%%%%%%%%%%%%%%%%%%%%%%%%%%%%%%%%%%%%%%%%%%%%%%%%%%%%%%%%%%%%%%%%%%%%%%%%%%%%%%%%%%%%%%%%%%%%%%%%%%%%%%%%%%%%%%%%%%%%%%%%%%%%%%%%%%%%%%%%%%%%%%%%%%%%%%%%%%%%%%%%%%%%%%%%%%%%%%%%%%%%%%%%%%%%%%%%%%%%%%%%%%%%%%%%%%%%%%%%%%%%%%%%%%%%%%%%%%%%%%%%%%%%%%%%%%%
\section{Introduction}
Let $\Gamma \subset \mathbb{R}^2$ be a smooth nonintersecting open arc, and we assume that $\Gamma$ can be extended to an arbitrary smooth, simply connected, closed curve $\partial \Omega$ enclosing a bounded domain $\Omega$ in $\mathbb{R}^2$. Let $k>0$ be the wave number, and let $\theta \in \mathbb{S}^{1}$ be incident direction, where $\mathbb{S}^{1}=\{x \in \mathbb{R}^2 : |x|=1 \}$ denotes the unit sphere in $\mathbb{R}^2$. We consider the following direct scattering problem: For $\theta \in \mathbb{S}^{1}$ determine $u^s$ such that
\begin{equation}
\Delta u^s+k^2u^s=0 \ \mathrm{in} \ \mathbb{R}^2\setminus  \Gamma, \label{1.1}
\end{equation}
\begin{equation}
u^s=-\mathrm{e}^{ik\theta \cdot x} \ \mathrm{on} \ \Gamma \label{1.2}
\end{equation}
\begin{equation}
\lim_{r \to \infty} \sqrt{r} \biggl( \frac{\partial u^{s}}{\partial  r}-iku^s \biggr)=0, \label{1.3}
\end{equation}
where $r=|x|$, and (\ref{1.3}) is the {\it Sommerfeld radiation condition}. Precisely, this problem is understood in the variational form, that is, determine $u^s \in H^{1}_{loc}(\mathbb{R}^2\setminus \Gamma)$ satisfying $u^s\bigl|_{\Gamma}=-\mathrm{e}^{ik\theta \cdot x}$, the Sommerfeld radiation condition (\ref{1.3}), and 
\begin{equation}
\int_{\mathbb{R}^2 \setminus \Gamma} \bigl[ \nabla u^s \cdot \nabla \overline{\varphi}-k^2u^s\overline{\varphi} \bigr]dx=0, \label{1.4}
\end{equation}
for all $\varphi \in H^{1}(\mathbb{R}^2\setminus \Gamma)$, $\varphi \bigl|_{\Gamma}=0$, with compact support. Here,  $H^{1}_{loc}(\mathbb{R}^2\setminus \Gamma)=\{u : \mathbb{R}^2\setminus \Gamma \to \mathbb{C} : u \bigl|_{B\setminus \Gamma} \in H^{1}(B\setminus \Gamma)\ \mathrm{for\ all\ open\ balls}\ B \}$ denotes the local Sobolev space of one order.
\par
It is well known that there exists a unique solution $u^{s}$ and it has the following asymptotic behavior (see, e.g., \cite{D. Colton and R. Kress}):
\begin{equation}
u^s(x)=\frac{\mathrm{e}^{ikr}}{\sqrt{r}}\Bigl\{ u^{\infty}(\hat{x},\theta)+O\bigl(1/r \bigr) \Bigr\} , \ r \to \infty, \ \ \hat{x}:=\frac{x}{|x|}. \label{1.5}
\end{equation}
The function $u^{\infty}$ is called the {\it far field pattern} of $u^s$. With the far field pattern $u^{\infty}$, we define the {\it far field operator} $F :L^{2}(\mathbb{S}^{1}) \to L^{2}(\mathbb{S}^{1})$ by
\begin{equation}
Fg(\hat{x}):=\int_{\mathbb{S}^{1}}u^{\infty}(\hat{x},\theta)g(\theta)ds(\theta), \ \hat{x} \in \mathbb{S}^{1}. \label{1.6}
\end{equation}
The inverse scattering problem we consider in this paper is to reconstruct the unknown arc $\Gamma$ from the far field pattern $u^{\infty}(\hat{x},\theta)$ for all $\hat{x} \in \mathbb{S}^{1}$, all $\hat{x} \in \mathbb{S}^{1}$ with one $k>0$. In other words, given the far field operator $F$, reconstruct $\Gamma$.
\par
In order to solve such an inverse problem, we use the idea of the monotonicity method. The feature of this method is to understand the inclusion relation of an unknown onject and artificial one by comparing the data operator with some operator corresponding to an artificial object. For recent works of the monotonicity method, we refer to \cite{R. Griesmaier and B. Harrach, B. Harrach and V. Pohjola and M. Salo1, B. Harrach and V. Pohjola and M. Salo2, B. Harrach and M. Ullrich1, B. Harrach and M. Ullrich2, Lakshtanov and Lechleiter}.
\par
Our aim in this paper is to provide the following two theorems.  
\begin{thm}
Let $\sigma \subset \mathbb{R}^2$ be a smooth nonintersecting open arc. Then,
\begin{equation}
\sigma \subset \Gamma \ \ \ \  \Longleftrightarrow \ \ \ \  H^{*}_{\sigma}H_{\sigma}\leq_{\mathrm{fin}} -\mathrm{Re}F,\label{1.7}
\end{equation}
where the Herglotz operator $H_{\sigma}:L^{2}(\mathbb{S}^{1}) \to L^{2}(\sigma)$ is given by
\begin{equation}
H_{\sigma}g(x):=\int_{\mathbb{S}^{1}}\mathrm{e}^{ik\theta \cdot x}g(\theta)ds(\theta), \ x \in \sigma, \label{1.8}
\end{equation}
and the inequality on the right hand side in (\ref{1.7}) denotes that $-\mathrm{Re}F-H^{*}_{\sigma}H_{\sigma}$ has only finitely many negative eigenvalues, and the real part of an operator $A$ is self-adjoint operators given by $\mathrm{Re}(A):=\displaystyle \frac{1}{2}(A+A^{*})$.
\end{thm}
\begin{thm}
Let $B \subset \mathbb{R}^2$ be a bounded open set. Then,
\begin{equation}
\Gamma \subset B \ \ \ \  \Longleftrightarrow \ \ \ \ -\mathrm{Re}F \leq_{\mathrm{fin}} \tilde{H}^{*}_{\partial B}\tilde{H}_{\partial B}, \label{1.9}
\end{equation}
where $\tilde{H}_{\partial B}:L^{2}(\mathbb{S}^{1}) \to H^{1/2}(\partial B)$ is given by
\begin{equation}
\tilde{H}_{\partial B}g(x):=\int_{\mathbb{S}^{1}}\mathrm{e}^{ik\theta \cdot x}g(\theta)ds(\theta), \ x \in \partial B. \label{1.10}
\end{equation}
\end{thm}
Theorem 1.1 determines whether an artificial open arc $\sigma$ is contained in $\Gamma$ or not. While, Theorem 1.2 determines an artificial domain $B$ contain $\Gamma$. In two theorems we can understand $\Gamma$ from the inside and outside.
\par
This paper is organized as follows. In Section 2, we give a rigorous definition of the above inequality.  Furthermore, we recall the properties of the far field operator and technical lemmas which are useful to prove main results. In Section 3 and 4, we prove Theorem 1.1 and 1.2 respectively. In Section 5, we give numerical examples based on Theorem 1.1.
%%%%%%%%%%%%%%%%%%%%%%%%%%%%%%%%%%%%%%%%%%%%%%%%%%%%%%%%%%%%%%%%%%%%%%%%%%%%%%%%%%%%%%%%%%%%%%%%%%%%%%%%%%%%%%%%%%%%%%%%%%%%%%%%%%%%%%%%%%%%%%%%%%%%%%%%%%%%%%%%%%%%%%%%%%%%%%%%%%%%%%%%%%%%%%%%%%%%%%%%%%%%%%%%%%%%%%%%%%%%%%
\section{Preliminary}
First, we give a rigorous definition of the inequality in Theorems 1.1 and 1.2. 
\begin{dfn}
Let $A, B:X \to X$ be self-adjoint compact linear operators on a Hilbert space $X$. We write
\begin{equation}
 A\leq_{\mathrm{fin}} B,\label{2.1}
\end{equation}
if $B-A$ has only finitely many negative eigenvalues.
\end{dfn}
The following lemma was shown in Corollary 3.3 of \cite{B. Harrach and V. Pohjola and M. Salo2}.
\begin{lem}
Let $A, B:X \to X$ be self-adjoint compact linear operators on a Hilbert space $X$ with an inner product $\langle \cdot, \cdot \rangle$. Then, the following statements are equivalent:
\begin{description}
\item[(a)]
$A\leq_{\mathrm{fin}} B$
\item[(b)]
There exists a finite dimensional subspace $V$ in $X$ such that
\begin{equation}
\langle (B-A)v, v \rangle \geq0,\label{2.2}
\end{equation}
for all $v \in V^{\bot}$.
\end{description}
\end{lem}
Secondly, we define several operators in order to mention properties of the far field operator $F$. The data-to-pattern operator $G:H^{1/2}(\Gamma) \to L^{2}(\mathbb{S}^{1})$ is defined by
\begin{equation}
Gf:=v^{\infty},
\label{2.3}
\end{equation}
where $v^{\infty}$ is the far field pattern of a radiating solution $v$ (that is, $v$ satisfies the Sommerfeld radiation condition) such that   
\begin{equation}
\Delta v+k^2v=0 \ \mathrm{in} \ \mathbb{R}^2\setminus  \Gamma, \label{2.4}
\end{equation}
\begin{equation}
v=f \ \mathrm{on} \ \Gamma. \label{2.5}
\end{equation} 
The following lemma was given by the same argument in Lemma 1.13 of \cite{Kirsch and Grinberg}.
\begin{lem}
The data-to-pattern operator $G$ is compact and injective.
\end{lem}
We define the single layer boundary operator $S:\tilde{H}^{-1/2}(\Gamma) \to H^{1/2}(\Gamma)$ by
\begin{equation}
S\varphi(x):=\int_{\Gamma} \varphi(y)\Phi(x,y)ds(y), \ x \in \Gamma, \label{2.6}
\end{equation}
where $\Phi(x,y)$ denotes the fundamental solution to Helmholtz equation in $\mathbb{R}^2$, i.e., 
\begin{equation}
\Phi(x,y):= \displaystyle \frac{i}{4}H^{(1)}_0(k|x-y|), \ x \neq y. \label{2.7}
\end{equation}
Here, we denote by
\begin{equation}
H^{1/2}(\Gamma):= \{u\bigl|_{\Gamma} : u \in H^{1/2}(\partial\Omega) \},\label{2.8}
\end{equation}
\begin{equation}
\tilde{H}^{1/2}(\Gamma):= \{u \in H^{1/2}(\partial \Omega) : supp(u) \subset  \overline{\Gamma} \}, \label{2.9}
\end{equation}
and $H^{-1/2}(\Gamma)$ and $\tilde{H}^{-1/2}(\Gamma)$ the dual spaces of $\tilde{H}^{1/2}(\Gamma)$ and $H^{1/2}(\Gamma)$ respectively. Then, we have the following inclusion relation:
\begin{equation}
\tilde{H}^{1/2}(\Gamma) \subset H^{1/2}(\Gamma) \subset L^{2}(\Gamma) \subset \tilde{H}^{-1/2}(\Gamma)\subset H^{-1/2}(\Gamma).\label{2.10}
\end{equation}
For these details, we refer to \cite{McLean}. The following two Lemmas was shown in Section 3 of \cite{Kirsch and Ritter}.
\begin{lem}
\begin{description}
\item[(a)] $S$ is an isomorphism from $\tilde{H}^{-1/2}(\Gamma)$ onto $H^{1/2}(\Gamma)$.
\item[(b)] Let $S_{i}$ be the boundary integral operator $(\ref{2.6})$ corresponding to the wave number $k=i$. The operator $S_{i}$ is self-adjoint and coercive, i.e, there exists $c_0>0$ such that
\begin{equation}
\langle \varphi, S_{i} \varphi \rangle \geq c_0 \left\| \varphi \right\|_{\tilde{H}^{-1/2}(\Gamma)}^2 \ for \ all \ \varphi \in \tilde{H}^{-1/2}(\Gamma), \label{2.11}
\end{equation}
where $\langle \cdot, \cdot \rangle$ denotes the duality pairing in $\langle \tilde{H}^{-1/2}(\Gamma), H^{1/2}(\Gamma) \rangle$.
\item[(c)] $S-S_{i}$ is compact.
\item[(d)] There exists a self-adjoint and positive square root $S^{1/2}_{i}:L^{2}(\Gamma) \to L^{2}(\Gamma)$ of $S_{i}$ which can be extended such that $S^{1/2}_{i}:\tilde{H}^{-1/2}(\Gamma) \to L^{2}(\Gamma)$ is an isomorphism and $S^{1/2\ *}_{i}S^{1/2}_{i}=S_{i}.$
\end{description}
\end{lem}
\begin{lem}
The far field operator $F$ has the following factorization:
\begin{equation}
F=-GS^{*}G^{*}.\label{2.12}
\end{equation}
where $G^{*}:L^{2}(\mathbb{S}^{1})\to \tilde{H}^{-1/2}(\Gamma)$ and $S^{*}:\tilde{H}^{-1/2}(\Gamma) \to H^{1/2}(\Gamma)$ are the adjoints of $G$ and $S$, respectively.
\end{lem}
Thirdly, we recall the following technical lemmas which will be useful to prove Theorems 1.1 and 1.2. We refer to Lemma 4.6 and 4.7 in \cite{B. Harrach and V. Pohjola and M. Salo2}.
\begin{lem}
Let $X$, $Y$, and $Z$ be Hilbert spaces, and let $A:X \to Y$ and $B:X \to Z$ be bounded linear operators. Then,
\begin{equation}
\exists C>0: \ \left\| Ax \right\|^2 \leq  C\left\| Bx \right\|^2 \ for \ all \ x \in X \ \ \ \ \Longleftrightarrow \ \ \ \ \mathrm{Ran}(A^{*})\subseteq \mathrm{Ran}(B^{*}).\label{2.13}
\end{equation}
\end{lem}
\begin{lem}
Let $X$, $Y$, $V \subset Z$ be subspaces of a vector space $Z$. If 
\begin{equation}
X\cap Y = \{ 0 \}, \ \ \ \ and \ \ \ \ X \subseteq Y+V,\label{2.14}
\end{equation}
then $\mathrm{dim}(X) \leq \mathrm{dim}(V)$.
\end{lem}
%%%%%%%%%%%%%%%%%%%%%%%%%%%%%%%%%%%%%%%%%%%%%%%%%%%%%%%%%%%%%%%%%%%%%%%%%%%%%%%%%%%%%%%%%%%%%%%%%%%%%%%%%%%%%%%%%%%%%%%%%%%%%%%%%%%%%%%%%%%%%%%%%%%%%%%%%%%%%%%%%%%%%%%%%%%%%%%%%%%%%%%%%%%%%%%%%%%%%%%%%%%%%%%%%%%%%%%%%%%%%%%%%%%%%%%%%%%%%%%%%%%%%%%%%%%%%%%%%%%%%%%%%%%%
%%%%%%%%%%%%%%%%%%%%%%%%%%%%%%%%%%%%%%%%%%%%%%%%%%%%%%%%%%%%%%%%%%%%%%%%%%%%%%%%%%%%%%%%%%%%%%%%%%%%%%%%%%%%%%%%%%%%%%%%%%%%%%%%%%%%%%%%%%%%%%%%%%%%%%%%%%%%%%%%%%%%%%%%%%%%%%%%%%%%%%%%%%%%%%%%%%%%%%%%%%%%%%%%%%%%%%%%%%%%%%%%%%%%%%%%%%%%%%%%%%%%%%%%%%%%%%%%%%%%%%%%%%%%%%%%
\section{Proof of Theorem 1.1}
In Section 3, we will show Theorem 1.1. Let $\sigma \subset \Gamma$. We denote by $R:L^{2}(\Gamma)\to L^{2}(\sigma)$ the restriction operator, $J:H^{1/2}(\Gamma)\to L^{2}(\Gamma)$ the compact embedding, and $H:L^{2}(\mathbb{S}^{1}) \to L^{2}(\Gamma)$, $\hat{H}:L^{2}(\mathbb{S}^{1}) \to H^{1/2}(\Gamma)$ the Herglotz operator, respectively. Then by these definitions and $\hat{H}^{*}=GS$, we have
\begin{eqnarray}
H_{\sigma}
&=&RH
\nonumber\\
&=&RJ\hat{H}
\nonumber\\
&=&RJS^{*}G^{*}.\label{3.1}
\end{eqnarray}
Using (\ref{3.1}) and Lemmas 2.4 and 2.5, $-\mathrm{Re}F-H^{*}_{\sigma}H_{\sigma}$ has the following factorization:
\begin{eqnarray}
-\mathrm{Re}F-H^{*}_{\sigma}H_{\sigma}
&=&G\bigl[\mathrm{Re}S- SJ^{*}R^{*}RJS^{*} \bigr]G^{*}
\nonumber\\
&=&G\bigl[S_{i}+\mathrm{Re}(S-S_{i})-SJ^{*}R^{*}RJS^{*} \bigr]G^{*}
\nonumber\\
&=&\bigl[GW^{*}\bigr]W^{*\ -1}\bigl[S_{i}+\mathrm{Re}(S-S_{i})-SJ^{*}R^{*}RJS^{*} \bigr]W^{-1} \bigl[GW^{*}\bigr]^{*}
\nonumber\\
&=&\bigl[GW^{*}\bigr]\bigl[I_{L^{2}(\Gamma)}+K \bigr]\bigl[GW^{*}\bigr]^{*},
\label{3.2}
\end{eqnarray}
where $W:=S^{1/2}_{i}: \tilde{H}^{-1/2}(\Gamma) \to L^{2}(\Gamma)$ is an extension of the square root of $S^{1/2}_{i}$, $K:=W^{*\ -1}\bigl[\mathrm{Re}(S-S_{i})-SJ^{*}R^{*}RJS^{*} \bigr]W^{-1}:L^{2}(\Gamma) \to L^{2}(\Gamma)$ is self-adjoint compact, and $I_{L^{2}(\Gamma)}$ is the identity operator on $L^{2}(\Gamma)$. Let $V$ be the sum of eigenspaces of $K$ associated to eigenvalues less than $-1/2$. Then, $V$ is a finite dimensional and
\begin{equation}
\langle (I_{L^{2}(\Gamma)}+K \bigr)v, v \rangle \geq0,\label{3.3}
\end{equation}
for all $v \in V^{\bot}$. Since for $g \in L^{2}(\mathbb{S}^{1})$
\begin{equation}
\bigl[GW^{*}\bigr]^{*}g \in V^{\bot} \ \ \ \  \Longleftrightarrow \ \ \ \ g \in [(GW^{*})V\bigr]^{\bot} ,\label{3.4}
\end{equation}
and $\mathrm{dim}[(GW^{*})V\bigr] \leq \mathrm{dim}(V)< \infty$, we have by (\ref{3.3}) and Lemma 2.2 that $H^{*}_{\sigma}H_{\sigma}\leq_{\mathrm{fin}} -\mathrm{Re}F$.
\par
Let now $\sigma \not\subset \Gamma$ and assume on the contrary $H^{*}_{\sigma}H_{\sigma}\leq_{\mathrm{fin}} -\mathrm{Re}F$, that is, by Lemma 2.2 there exists a finite dimensional subspace $V$ in $L^{2}(\mathbb{S}^{1})$ such that
\begin{equation}
\langle (-\mathrm{Re}F-H^{*}_{\sigma}H_{\sigma})v, v \rangle \geq 0,\label{3.5}
\end{equation}
for all $v \in V^{\bot}$. Since $\sigma \not\subset \Gamma$,  we can take a small open arc $\sigma_0 \subset \sigma$ such that $\sigma_0 \cap \Gamma= \emptyset$, which implies that for all $v \in V^{\bot}$ 
\begin{eqnarray}
\left\| H_{\sigma_0}v \right\|^{2}_{L^{2}(\sigma_0)}
&\leq&
\left\| H_{\sigma}v \right\|^{2}_{L^{2}(\sigma)}
\nonumber\\
&\leq&\langle (-\mathrm{Re}F)v, v \rangle_{L^{2}(\mathbb{S}^{1})}
\nonumber\\
&=&\langle (\mathrm{Re}S^{*})G^{*}v, G^{*}v \rangle
\nonumber\\
&\leq&\left\| \mathrm{Re}S^{*} \right\| \left\|G^{*}v \right\|^{2}.
\label{3.6}
\end{eqnarray}
Before showing contradiction with (\ref{3.6}), we will show the following lemma.
\begin{lem}
\begin{description}
\item[(a)]
$\mathrm{dim}(\mathrm{Ran}(H_{\sigma_0}^{*}))=\infty$
\item[(b)]
$\mathrm{Ran}(G)\cap \mathrm{Ran}(H_{\sigma_0}^{*})=\{ 0 \}$.
\end{description}
\end{lem}
\begin{proof}[{\bf Proof of Lemma 3.1}]
{\bf (a)} By the same argument in (\ref{3.1}) we have
\begin{equation}
H_{\sigma_0}=J_{\sigma_0}\hat{H}_{\sigma_0}=J_{\sigma_0}S_{\sigma_0}^{*}G_{\sigma_0}^{*},\label{3.7}
\end{equation}
where $G_{\sigma_0}:H^{1/2}(\sigma_0) \to L^{2}(\mathbb{S}^{1})$, $S_{\sigma_0}:\tilde{H}^{-1/2}(\sigma_0) \to H^{1/2}(\sigma_0)$, and $J_{\sigma_0}:H^{1/2}(\sigma_0)\to L^{2}(\sigma_0)$ are the data-to-pattern operator, the single layer boundary operator, and the compact embedding, respectively corresponding to $\sigma_0$. Since $H_{\sigma_0}^{*}=G_{\sigma_0}S_{\sigma_0}J^{*}_{\sigma_0}$, $\mathrm{Ran}(J^{*}_{\sigma_0})$ is dense, and  $G_{\sigma_0}S_{\sigma_0}$ is injective, we have $\mathrm{dim}(\mathrm{Ran}(H_{\sigma_0}^{*}))=\mathrm{dim}(\mathrm{Ran}(J_{\sigma_0}^{*}))=\infty$.
\par
{\bf (b)} By (\ref{3.7}), we have $\mathrm{Ran}(H_{\sigma_0}^{*})\subset \mathrm{Ran}(G_{\sigma_0})$. Let $h \in \mathrm{Ran}(G)\cap \mathrm{Ran}(G_{\sigma_0})$, i.e., $h=v_{\Gamma}^{\infty}=v_{\sigma_0}^{\infty}$ where $v_{\Gamma}^{\infty}$ and $v_{\sigma_0}^{\infty}$ are far field patterns associated to scatterers $\Gamma$ and $\sigma_0$ respectively. Then by Rellich lemma and unique continuation we have $v_{\Gamma}=v_{\sigma_0}\ \mathrm{in} \ \mathbb{R}^2\setminus  (\Gamma \cup \sigma_0)$. Hence, we can define $v \in H^{1}_{loc}(\mathbb{R}^2)$ by
\begin{eqnarray}
v:=\left\{ \begin{array}{ll}
v_{\Gamma}=v_{\sigma_0} & \quad \mbox{in $\mathbb{R}^2\setminus  (\Gamma \cup \sigma_0)$}  \\
v_{\Gamma} & \quad \mbox{on $\sigma_0$} \\
v_{\sigma_0} & \quad \mbox{on $\Gamma$} \\
\end{array} \right.\label{3.8}
\end{eqnarray}
and $v$ is a radiating solution to
\begin{equation}
\Delta v+k^2v=0 \ \mathrm{in} \ \mathbb{R}^2. \label{3.9}
\end{equation}
Thus $v=0 \ \mathrm{in} \ \mathbb{R}^2$, which implies that $h=0$.   
\end{proof}
By the above lemma and using Lemma 2.7, we get
\begin{equation}
\mathrm{Ran}(H^{*}_{\sigma_0}) \not\subseteq \mathrm{Ran}(G)+V= \mathrm{Ran}(G,\ P_{V}),\label{3.10}
\end{equation}
where $P_{V}:L^{2}(\mathbb{S}^{1})\to L^{2}(\mathbb{S}^{1})$ is the orthognal projection on $V$. Lemma 2.6 implies that for any $C>0$ there exists a $v_c$ such that 
\begin{equation}
\left\|H_{\sigma_0}v_{c} \right\|^2 >  C^{2}\left\| \left(
    \begin{array}{ccc}
      G^{*} \\
      P_{V} \\
    \end{array}
\right) v_c \right\|^2=C^{2}(\left\|G^{*}v_{c} \right\|^2+\left\|P_{V}v_{c} \right\|^2).\label{3.11}
\end{equation}
Hence, there exists a sequence $(v_m)_{m\in \mathbb{N}} \subset  L^{2}(\mathbb{S}^{1})$ such that $\left\|H_{\sigma_0}v_{m} \right\| \to \infty$ and $\left\|G^{*}v_{m} \right\|^2+\left\|P_{V}v_{m} \right\| \to 0$ as $m \to \infty$. Setting $\tilde{v}_{m}:=v_{m}-P_{V}v_{m} \in V^{\bot}$ we have as $m\to \infty$,
\begin{equation}
\left\|H_{\sigma_0}\tilde{v}_{m} \right\| \geq \left\|H_{\sigma_0}v_{m} \right\|-\left\|H_{\sigma_0}\right\|\left\|P_{V}v_{m}\right\| \to \infty,\label{3.12}
\end{equation}
\begin{equation}
\left\|G^{*}\tilde{v}_{m} \right\| \leq \left\|G^{*}v_{m} \right\|+\left\| G^{*} \right\| \left\|P_{V}v_{m} \right\| \to 0.\label{3.13}
\end{equation}
This contradicts (\ref{3.6}). Therefore, we have $H^{*}_{\sigma}H_{\sigma}\not\leq_{\mathrm{fin}} -\mathrm{Re}F$. Theorem 1.1 has been shown. \qed
%%%%%%%%%%%%%%%%%%%%%%%%%%%%%%%%%%%%%%%%%%%%%%%%%%%%%%%%%%%%%%%%%%%%%%%%%%%%%%%%%%%%%%%%%%%%%%%%%%%%%%%%%%%%%%%%%%
%%%%%%%%%%%%%%%%%%%%%%%%%%%%%%%%%%%%%%%%%%%%%%%%%%%%%%%%%%%%%%%%%%%%%%%%%%%%%%%%%%%%%%%%
%%%%%%%%%%%%%%%%%%%%%%%%%%%%%%%%%%%%%%%%%%%%%%%%%%%%%%%%%%%%%%%%%%%%%%%%%%%%%%%%%%%%%%%%%%%%%%%%%%%%%%%%%%%%%%%%%%%%%%%%%%%%%%%%%%%%%%%%%%%%%%%%%%%%%%%%%%%%%%%
\section{Proof of Theorem 1.2}
In Section 4, we will show Theorem 1.2. Let $\Gamma \subset B$. We denote by $G_{\partial B}:H^{1/2}(\partial B) \to L^{2}(\mathbb{S}^{1})$ and $S_{\partial B}:\tilde{H}^{-1/2}(\sigma_0) \to H^{1/2}(\sigma_0)$ are the data-to-pattern operator and the single layer boundary operator, respectively corresponding to closed curve $\partial B$. They have  the same properties like Lemma 2.3 and 2.4 and we have $\tilde{H}_{\partial B}^{*}=G_{\partial B}S_{\partial B}$. (See, e.g., \cite{Kirsch and Grinberg}.) We define $T:H^{1/2}(\Gamma)\to H^{1/2}(\partial B)$ by 
\begin{equation}
Tf:=v\bigl|_{\partial B},
\label{4.1}
\end{equation}
where $v$ is a radiating solution such that   
\begin{equation}
\Delta v+k^2v=0 \ \mathrm{in} \ \mathbb{R}^2\setminus  \Gamma, \label{4.2}
\end{equation}
\begin{equation}
v=f \ \mathrm{on} \ \Gamma. \label{4.3}
\end{equation} 
$T$ is compact since its mapping is from $H^{1/2}(\Gamma)$ to $C^{\infty}(\partial B)$. Furthermore, by the definition of $T$ we have that $G=G_{\partial B}T$. Thus, we have
\begin{eqnarray}
\tilde{H}^{*}_{\partial B}\tilde{H}_{\partial B}+\mathrm{Re}F 
&=&G_{\partial B}S_{\partial B}S^{*}_{\partial B}G^{*}_{\partial B}+G_{\partial B}\bigl[-T\mathrm{Re}(S)T^{*} \bigr]G_{\partial B}^{*}
\nonumber\\
&=&G_{\partial B}\bigl[S_{\partial B, i}S^{ * }_{\partial B, i}+ K \bigr]G_{\partial B}^{*}
\nonumber\\
&=&\bigl[G_{\partial B}W^{*}\bigr]\bigl[W^{*\ -1}S_{\partial B,i}S^{ *}_{\partial B,i}W^{-1}+ K'\bigr]\bigl[G_{\partial B}W^{*}\bigr]^{*}, \ \ \ \ \ \ \ \ \  \label{4.4}
\end{eqnarray}
where $K$ and $K'$ are some self-adjoint compact operators, and $W:=S^{1/2}_{\partial B,i}:H^{-1/2}(\partial B) \to L^{2}(\partial B)$ is an extension of the square root of $S_{\partial B,i}$. Let $V$ be the sum of eigenspaces of $K'$ associated to eigenvalues less than $-\frac{1}{2}\left\| S^{ *}_{\partial B,i}W^{-1} \right\|^{-2}$. Then $V$ is a finite dimensional, and for all $g \in \bigl[(G_{\partial B}W^{*})V\bigr]^{\bot}$ we have
\begin{eqnarray}
&&\langle (\tilde{H}^{*}_{\partial B}\tilde{H}_{\partial B}+\mathrm{Re}F)g, g \rangle
\nonumber\\
&=&
\left\| (S^{ *}_{\partial B,i}W^{-1})\bigl[G_{\partial B}W^{*}\bigr]^{*}g \right\|^{2}_{H^{1/2}(\partial B)}+\langle K'\bigl[G_{\partial B}W^{*}\bigr]^{*}g, \bigl[G_{\partial B}W^{*}\bigr]^{*}g \rangle_{L^{2}(\partial B)}
\nonumber\\
&\geq&\left\| (S^{ *}_{\partial B,i}W^{-1})^{-1}\right\|^{-2} \left\| \bigl[G_{\partial B}W^{*}\bigr]^{*}g \right\|^{2}-\frac{1}{2}\left\| (S^{ *}_{\partial B,i}W^{-1})^{-1} \right\|^{-2}\left\| \bigl[G_{\partial B}W^{*}\bigr]^{*}g \right\|^{2}
\nonumber\\
&\geq& 0.
\label{4.5}
\end{eqnarray}
Therefore, $-\mathrm{Re}F \leq_{\mathrm{fin}} \tilde{H}^{*}_{\partial B}\tilde{H}_{\partial B}$.
\par
Let now $\Gamma \not\subset B$ and assume on the contrary $-\mathrm{Re}F \leq_{\mathrm{fin}} \tilde{H}^{*}_{\partial B}\tilde{H}_{\partial B}$, i.e., by Lemma 2.2 there exists a finite dimensional subspace $V$ in $L^{2}(\mathbb{S}^{1})$ such that
\begin{equation}
\langle (\tilde{H}^{*}_{\partial B}\tilde{H}_{\partial B}+\mathrm{Re}F)v, v \rangle \geq 0,\label{4.6}
\end{equation}
for all $v \in V^{\bot}$. Since $\Gamma \not\subset B$,  we can take a small open arc $\Gamma_0 \subset \Gamma$ such that $\Gamma_0 \cap B= \emptyset$. We define $L:H^{1/2}(\Gamma_0)\to H^{1/2}(\Gamma)$ by
\begin{equation}
Lf:=v\bigl|_{\Gamma},
\label{4.7}
\end{equation}
where $v$ is a radiating solution such that   
\begin{equation}
\Delta v+k^2v=0 \ \mathrm{in} \ \mathbb{R}^2\setminus  \Gamma_0, \label{4.8}
\end{equation}
\begin{equation}
v=f \ \mathrm{on} \ \Gamma_0. \label{4.9}
\end{equation} 
By the definition of $L$, we have $G_{\Gamma_0}=GL$. By $\tilde{H}_{\Gamma_0}=S^{*}_{\Gamma_0}G^{*}_{\Gamma_0}$, we have
\begin{eqnarray}
\left\| H_{\Gamma_0}x \right\|^{2}_{L^{2}(\Gamma_0)}
&\leq&
\left\|  \tilde{H}_{\Gamma_0}x \right\|^{2}_{H^{1/2}(\Gamma_0)}
\nonumber\\
&\leq&
\left\| S^{*}_{\Gamma_0} \right\|^{2} \left\| G^{*}_{\Gamma_0}x \right\|^{2}
\nonumber\\
&\leq&
\left\| S^{*}_{\Gamma_0} \right\|^{2} \left\| L^{*}\right\|^{2} \left\|G^{*}x \right\|^{2},\label{4.10}
\end{eqnarray}
for $x \in L^{2}(\mathbb{S}^{1})$. Since $\mathrm{Re}S$ is of the form $\mathrm{Re}S=S_{i}+\mathrm{Re}(S-S_i)$, by the similar argument in (\ref{3.2})--(\ref{3.3}) and (\ref{4.4})--(\ref{4.5}), there exists a finite dimensional subspace $W$ in $L^{2}(\mathbb{S}^{1})$ such that for $x \in W^{\bot}$
\begin{eqnarray}
\left\| G^{*}x \right\|^{2}\leq
C\langle (\mathrm{Re}S)G^{*}x, G^{*}x \rangle
=C\langle (-\mathrm{Re}F)x, x \rangle.
\label{4.11}
\end{eqnarray}
Collecting (\ref{4.6}), (\ref{4.10}), and (\ref{4.11}) we have
\begin{eqnarray}
\left\| H_{\Gamma_{0}}x \right\|^{2}
&\leq& C\langle (-\mathrm{Re}F)x, x \rangle
\leq 
C\left\| \tilde{H}_{\partial B}x \right\|^{2} 
\nonumber\\
&\leq& C\left\| S^{*}_{\partial B} \right\|^{2} \left\| G^{*}_{\partial B}x \right\|_{H^{-1/2}(\partial B)}^{2}.
\label{4.12}
\end{eqnarray}
for $x \in (V\cup W)^{\bot}$. 
\begin{lem}
\begin{description}
\item[(a)]
$\mathrm{dim}(\mathrm{Ran}(H_{\Gamma_0}^{*}))=\infty$
\item[(b)]
$\mathrm{Ran}(G_{\partial B})\cap \mathrm{Ran}(H_{\Gamma_0}^{*})=\{ 0 \}$.
\end{description}
\end{lem}
\begin{proof}[{\bf Proof of Lemma 4.1}]
(a) is given by the same argument in Lemma 3.1. 
\par
{\bf (b)} By (\ref{3.7}), we have $\mathrm{Ran}(H_{\Gamma_0}^{*})\subset \mathrm{Ran}(G_{\Gamma_0})$. Let $h \in \mathrm{Ran}(G_{\partial B})\cap \mathrm{Ran}(G_{\Gamma_0})$, i.e., $h=v_{B}^{\infty}=v_{\Gamma_0}^{\infty}$ where $v_{B}^{\infty}$ and $v_{\Gamma_0}^{\infty}$ are far field patterns associated to scatterers $B$ and $\Gamma_0$ respectively. Then by Rellich lemma and unique continuation we have $v_{B}=v_{\Gamma_0}\ \mathrm{in} \ \mathbb{R}^2\setminus  (B \cup \Gamma_0)$. Hence, we can define $v \in H^{1}_{loc}(\mathbb{R}^2)$ by
\begin{eqnarray}
v:=\left\{ \begin{array}{ll}
v_{B}=v_{\Gamma_0} & \quad \mbox{in $\mathbb{R}^2\setminus  (B \cup \Gamma_0)$}  \\
v_{\Gamma_0} & \quad \mbox{on $B$} \\
v_{B} & \quad \mbox{on $\Gamma_0$} \\
\end{array} \right.\label{4.13}
\end{eqnarray}
and $v$ is a radiating solution to
\begin{equation}
\Delta v+k^2v=0 \ \mathrm{in} \ \mathbb{R}^2. \label{4.14}
\end{equation}
Thus $v=0 \ \mathrm{in} \ \mathbb{R}^2$, which implies that $h=0$.   
\end{proof}
By the above lemma and using Lemma 2.7, we get
\begin{equation}
\mathrm{Ran}(H^{*}_{\Gamma_0}) \not\subseteq \mathrm{Ran}(G_{\partial B})+(V\cup W)= \mathrm{Ran}(G_{\partial B},\ P_{V\cup W}),\label{4.15}
\end{equation}
where $P_{V\cup W}:L^{2}(\mathbb{S}^{1})\to L^{2}(\mathbb{S}^{1})$ is the orthognal projection on $V\cup W$. Lemma 2.6 implies that for any $C>0$ there exists a $x_c$ such that 
\begin{equation}
\left\|H_{\Gamma_0}x_{c} \right\|^2 >  C^{2}\left\| \left(
    \begin{array}{ccc}
      G^{*}_{\partial B} \\
      P_{V\cup W} \\
    \end{array}
\right) x_c \right\|^2=C^{2}(\left\|G^{*}_{\partial B}x_{c} \right\|^2+\left\|P_{V\cup W}x_{c} \right\|^2).\label{4.16}
\end{equation}
Hence, there exists a sequence $(x_m)_{m\in \mathbb{N}} \subset  L^{2}(\mathbb{S}^{1})$ such that $\left\|H_{\Gamma_0}x_{m} \right\| \to \infty$ and $\left\|G^{*}_{\partial B}x_{m} \right\|^2+\left\|P_{V\cup W}x_{m} \right\| \to 0$ as $m \to \infty$. Setting $\tilde{x}_{m}:=x_{m}-P_{V\cup W}x_{m} \in (V\cup W)^{\bot}$ we have as $m\to \infty$,
\begin{equation}
\left\|H_{\Gamma_0}\tilde{x}_{m} \right\| \geq \left\|H_{\Gamma_0}x_{m} \right\|-\left\|H_{\Gamma_0}\right\|\left\|P_{V\cup W}x_{m}\right\| \to \infty,\label{4.17}
\end{equation}
\begin{equation}
\left\|G^{*}_{\partial B}\tilde{x}_{m} \right\| \leq \left\|G^{*}_{\partial B}x_{m} \right\|+\left\| G^{*}_{\partial B} \right\| \left\|P_{V\cup W}x_{m} \right\| \to 0.\label{4.18}
\end{equation}
This contradicts (\ref{4.12}). Therefore, we have $  -\mathrm{Re}F  \not\leq_{\mathrm{fin}} \tilde{H}^{*}_{\partial B}\tilde{H}_{\partial B}$. Theorem 1.2 has been shown. \qed
%%%%%%%%%%%%%%%%%%%%%%%%%%%%%%%%%%%%%%%%%%%%%%%%%%%%%%%%%%%%%%%%%%%%%%%%%%%%%%%%%%%%%%%%%%%%%%%%%%%%%%%%%%%%%%%%%%%%%%%%%%%%%%%%%%%%%%%%%%%%%%%%%%%%%%%%%%%%%%%%%%%%%%%%%%%%%%%%%%%%%%%%%%%%%%%%%%%%%%%%%%%%%%%%%%%%%%%%%%%%%%%%%%%%%%%%%%%%%%%%%%%%%%%%%%%%%%%%%%%%%%%%%%%%%%%%%%%%%%%%%%%%%%%%%%%%%%%%%%%%%%%%%%%%%%
%%%%%%%%%%%%%%%%%%%%%%%%%%%%%%%%%%%%%%%%%%%%%%%%%%%%%%%%%%%%%%%%%%%%%%%%%%%%%%%%%%%%%%%%%%%%%%%%%%%%%%%%%%%%%%%%%%%%%%%%%%%%%%%%%%%%%%%%%%%%%%%%%%%%%%%%%%%%%%%%%%%%%%%%%%%%%%%%%%%%%%%%%%%%%%%%%%%%%%%%%%%%%%%%%%%%%%%%%%%%%%%%%%%%%%%%%%%%%%%%%%%%%%%%%%%%%%%%%%%%%%%%%%%%%%%%%%%%%%%%%%%%%%%%%%%%%%%%%%%%%%%%
\section{Numerical examples}
In Section 5, we discuss the numerical examples based on Theorem 1.1. The following three open arcs $\Gamma_j$ ($j=1,2,3$) are considered (see Figure 1):
\begin{description}
  \item[(a)] $\Gamma_1=\left\{(s, s) | -1\leq s \leq 1 \right\}$
  \item[(b)] $\Gamma_2=\left\{\biggl(2\mathrm{sin}\Bigl(\frac{\pi}{8}+(1+s)\frac{3\pi}{8} \Bigr) -\frac{2}{3}, \ \mathrm{sin}\Bigl(\frac{\pi}{4}+(1+s)\frac{3\pi}{4} \biggr) \biggl| -1\leq s \leq 1 \right\}$
  \item[(c)] $\Gamma_3=\left\{\biggl(s, \ \mathrm{sin}\Bigl(\frac{\pi}{4}+(1+s)\frac{3\pi}{4} \biggr) \biggl| -1\leq s \leq 1 \right\}$
\end{description}

%Figure1
\vspace{5mm}
\begin{figure}[h]
 \begin{minipage}[b]{0.325\linewidth}
  \centering
  \includegraphics[keepaspectratio, scale=0.15]
  {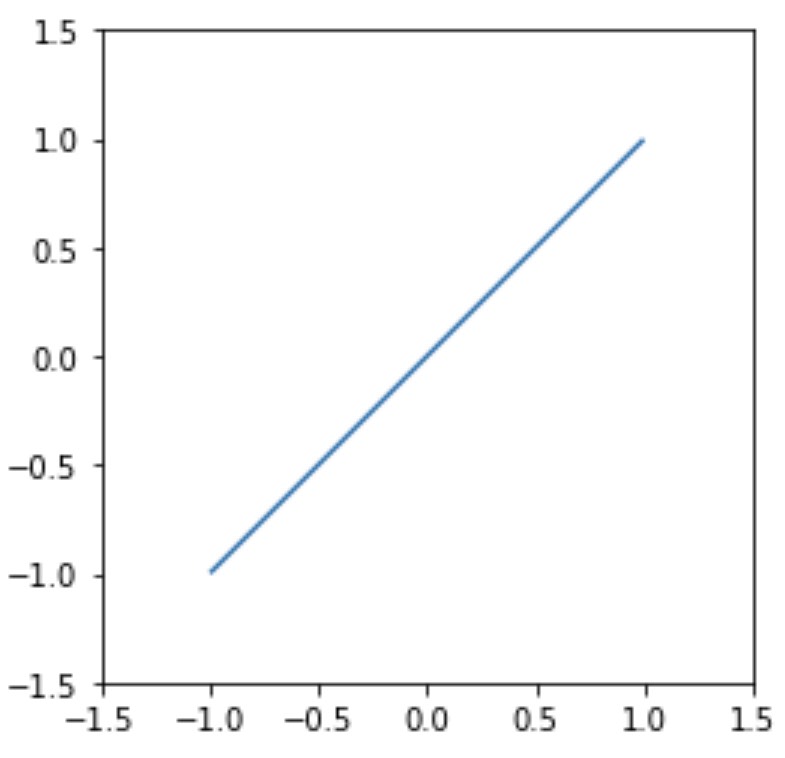}
  \subcaption{$\Gamma_1$}
 \end{minipage}
 \begin{minipage}[b]{0.325\linewidth}
  \centering
  \includegraphics[keepaspectratio, scale=0.15]
  {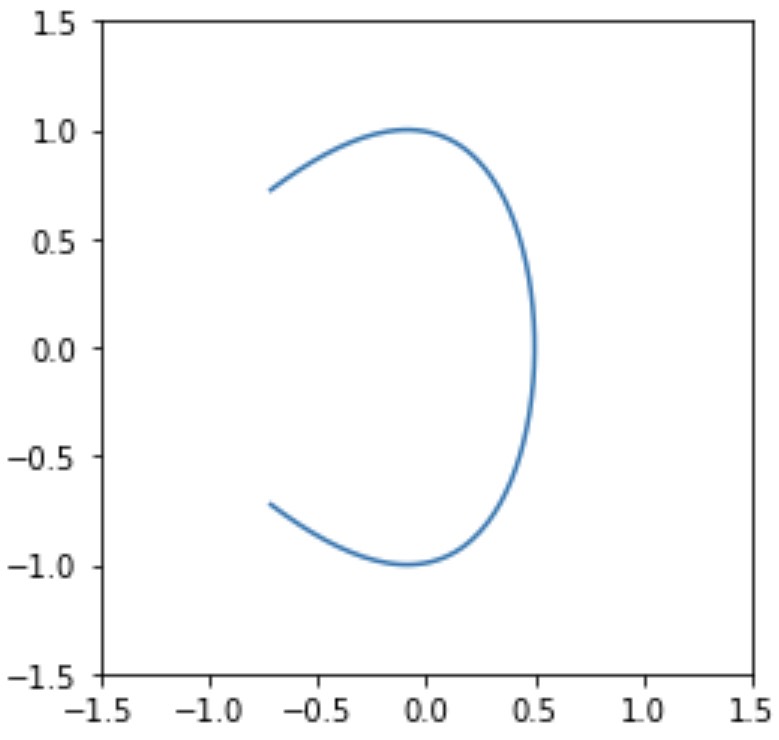}
  \subcaption{$\Gamma_2$}
 \end{minipage}
 \begin{minipage}[b]{0.325\linewidth}
  \centering
  \includegraphics[keepaspectratio, scale=0.15]
  {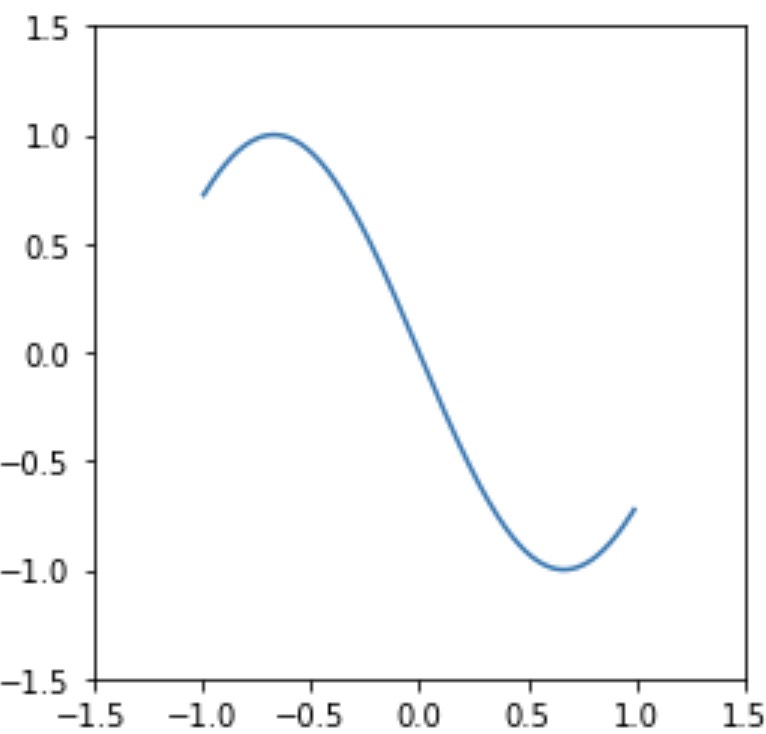}
  \subcaption{$\Gamma_3$}
 \end{minipage}
 \caption{The original open arc}
\end{figure}
\vspace{5mm}

Based on Theorem 1.1, the indicator function in our examples is given by 
\begin{equation}
I(\sigma):= \# \left\{\mathrm{negative} \ \mathrm{eigenvalues}  \ \mathrm{of} -\mathrm{Re}F-H^{*}_{\sigma}H_{\sigma} \right\}
\end{equation}
The idea to reconstruct $\Gamma_j$ is to plot the value of $I(\sigma)$ for many of small $\sigma$ in the sampling region. Then, we expect from Theorem 1.1 that the value of the function $I(\sigma)$ is low if $\sigma$ is close to $\Gamma_j$. 
\par
Here, $\sigma$ is chosen in two ways; One is the vertical line segment $\sigma^{ver}_{i,j}:=z_{i,j}+\{0 \} \times [-\frac{R}{2M}, \frac{R}{2M}]$ where $z_{i,j}:=(\frac{Ri}{M}, \frac{Rj}{M})$ ($i,j = -M, -M+1, ..., M$) denote the center of $\sigma^{ver}_{i,j}$, and $\frac{R}{M}$ is the length of $\sigma^{ver}_{i,j}$, and $R>0$ is length of sampling square region $[-R, R]^2$, and $M \in \mathbb{N}$ is large to take a small segment. The other is horizontal one $\sigma^{hor}_{i,j}:=z_{i,j}+[-\frac{R}{2M}, \frac{R}{2M}] \times \{0 \}$.
\par
The far field operator $F$ is approximated by the matrix
\begin{equation}
F \approx \frac{2\pi}{N} \bigl(u^{\infty}(\hat{x}_l, \theta_m) \bigr)_{1 \leq l,m \leq N} \in \mathbb{C}^{N \times N}
\end{equation}
where $\hat{x}_l=\bigl(\mathrm{cos}(\frac{2\pi l}{N}), \mathrm{sin}(\frac{2\pi l}{N}) \bigr)$ and $\theta_m=\bigl(\mathrm{cos}(\frac{2\pi m}{N}), \mathrm{sin}(\frac{2\pi m}{N}) \bigr)$. The far field pattern $u^{\infty}$ of the problem (\ref{1.1})--(\ref{1.3}) is computed by the Nystrom method in \cite{Kress}. The operator $H^{*}_{\sigma}H_{\sigma}$ is approximated by
\begin{equation}
H^{*}_{\sigma}H_{\sigma} \approx \frac{2\pi}{N} \biggl(\int_{\sigma}e^{iky\cdot(\theta_m-\hat{x}_l)}dy \biggr)_{1 \leq l,m \leq N} \in \mathbb{C}^{N \times N}
\end{equation}
When $\sigma$ is given by the vertical and horizontal line segment, we can compute the integrals
\begin{equation}
\int_{\sigma^{ver}_{i,j}}e^{iky\cdot(\theta_m-\hat{x}_l)}dy=\frac{R}{M}e^{ik(\theta_m-\hat{x}_l)\cdot z_{i,j}}\mathrm{sinc}\biggl(\frac{kR}{2M\pi}\Bigl( \mathrm{sin}\bigl(\frac{2\pi m}{N}\bigr)-\mathrm{sin}\bigl(\frac{2\pi l}{N}\bigr) \bigr) \biggr)
\end{equation}
\begin{equation}
\int_{\sigma^{hor}_{i,j}}e^{iky\cdot(\theta_m-\hat{x}_l)}dy=\frac{R}{M}e^{ik(\theta_m-\hat{x}_l)\cdot z_{i,j}}\mathrm{sinc}\biggl(\frac{kR}{2M\pi}\Bigl( \mathrm{cos}\bigl(\frac{2\pi m}{N}\bigr)-\mathrm{cos}\bigl(\frac{2\pi l}{N}\bigr) \bigr) \biggr)
\end{equation}
\par
In our examples we fix $R=1.5$, $M=100$, $N=60$, and wavenumber $k=1$. The Figure 2 is given by plotting the values of the vertical indicator function 
\begin{equation}
I_{ver}(z_{i,j}):=I(\sigma^{ver}_{i,j})
\end{equation}
for each $i, j = -100, 99, ..., 100$. The Figure 3 is given by plotting the values of the horizontal indicator function 
\begin{equation}
I_{hor}(z_{i,j}):=I(\sigma^{hor}_{i,j})
\end{equation}
for each $i, j = -100, -99, ..., 100$. We obverse that $\Gamma_j$ seems to be reconstructed independently of the direction of linear segment.

%Figure 2
\begin{figure}[h]
 \begin{minipage}[b]{0.325\linewidth}
  \centering
  \includegraphics[keepaspectratio, scale=0.13]
  {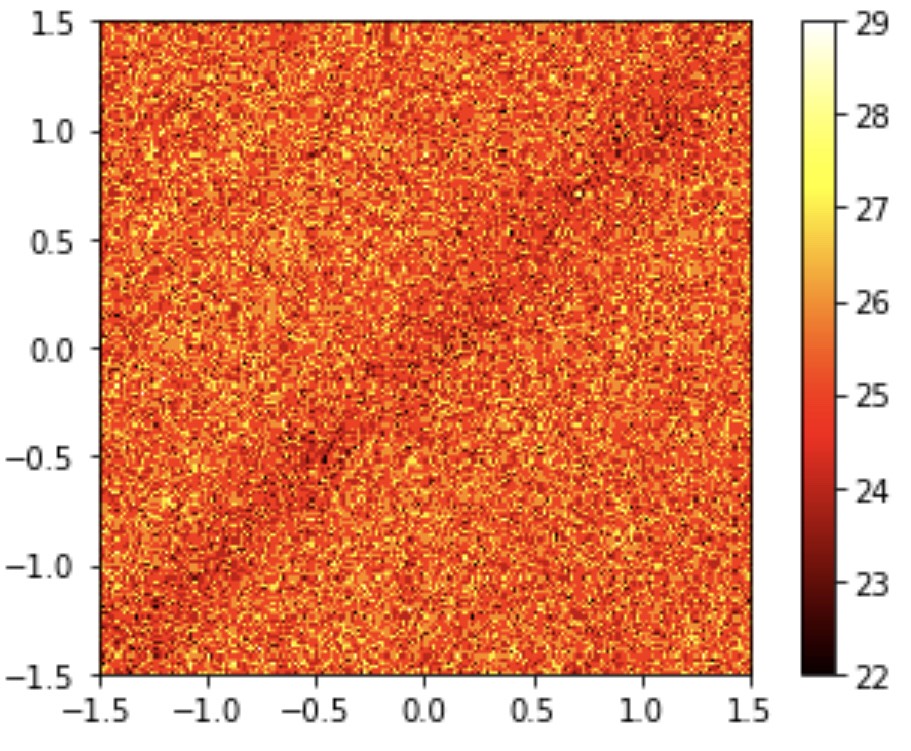}
  \subcaption{$\Gamma_1$}
 \end{minipage}
 \begin{minipage}[b]{0.325\linewidth}
  \centering
  \includegraphics[keepaspectratio, scale=0.13]
  {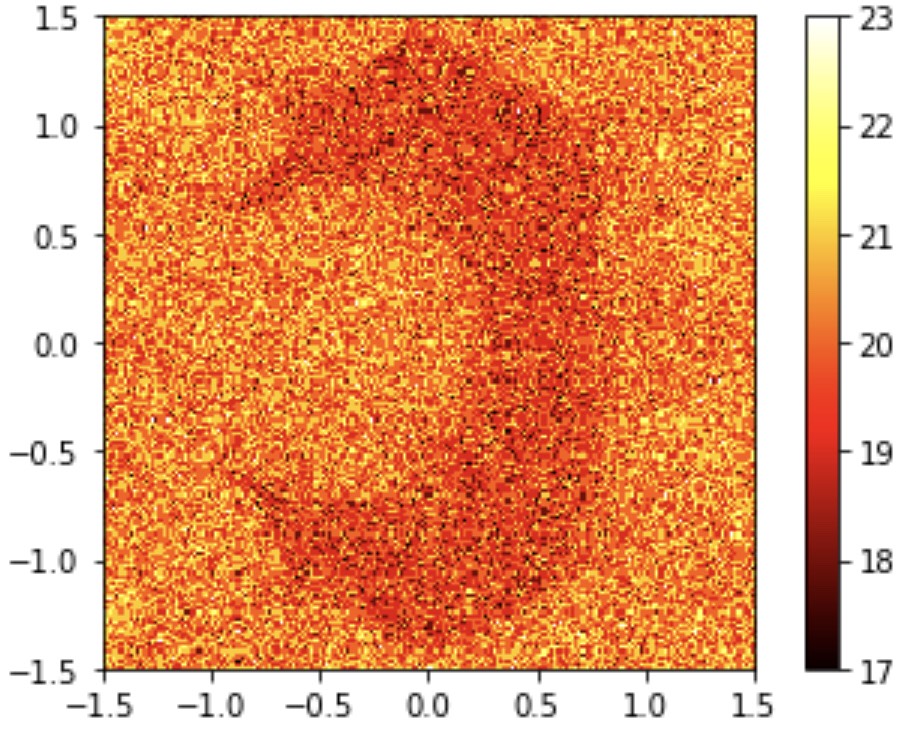}
  \subcaption{$\Gamma_2$}
 \end{minipage}
 \begin{minipage}[b]{0.325\linewidth}
  \centering
  \includegraphics[keepaspectratio, scale=0.13]
  {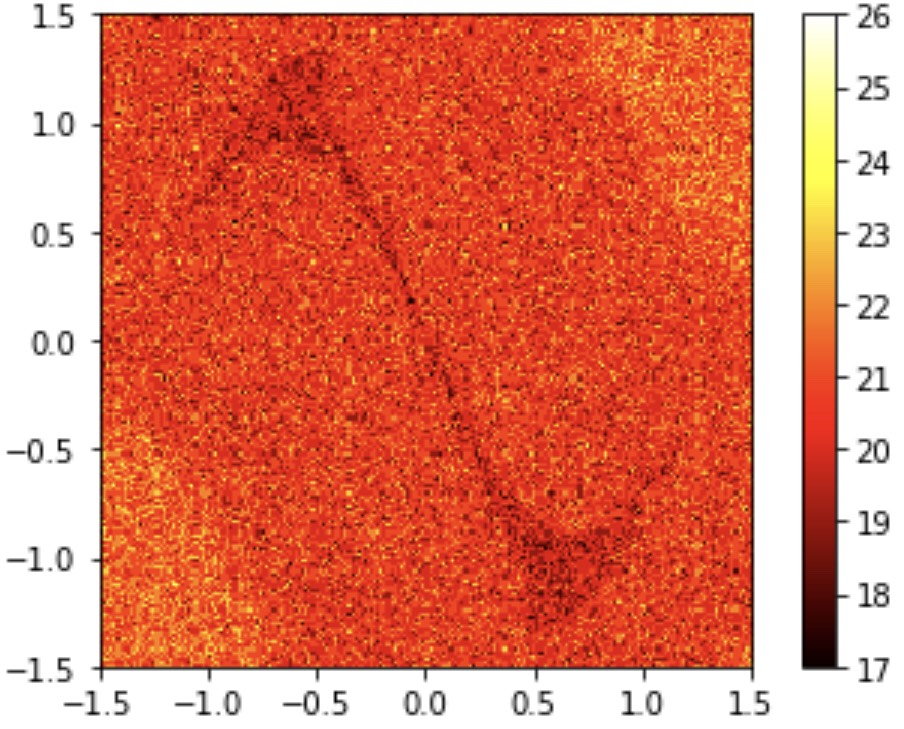}
  \subcaption{$\Gamma_3$}
 \end{minipage}
 \caption{Reconstruction by the vertical indicator function $I_{ver}$}
 \vspace{5mm}
\end{figure}

%Figure 3
\begin{figure}[h]
 \begin{minipage}[b]{0.325\linewidth}
  \centering
  \includegraphics[keepaspectratio, scale=0.13]
  {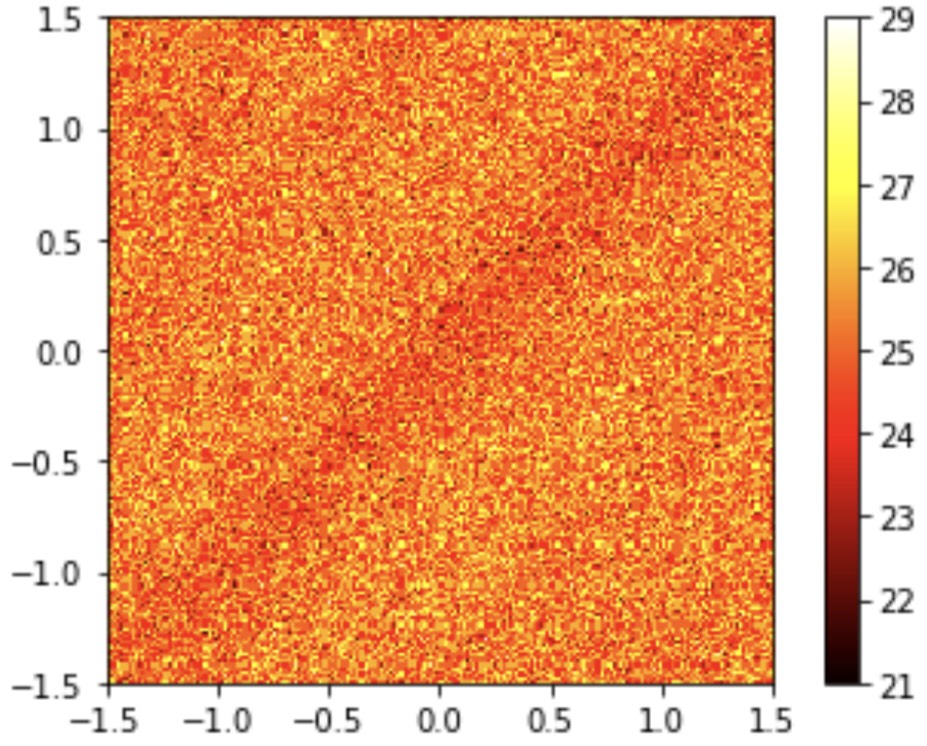}
  \subcaption{$\Gamma_1$}
 \end{minipage}
 \begin{minipage}[b]{0.325\linewidth}
  \centering
  \includegraphics[keepaspectratio, scale=0.13]
  {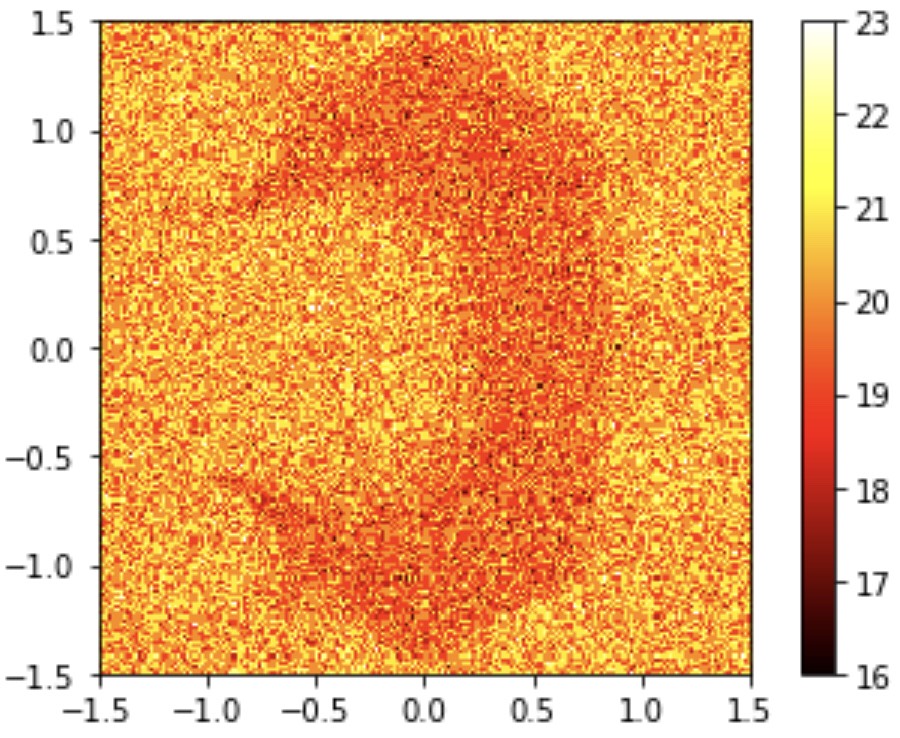}
  \subcaption{$\Gamma_2$}
 \end{minipage}
 \begin{minipage}[b]{0.325\linewidth}
  \centering
  \includegraphics[keepaspectratio, scale=0.13]
  {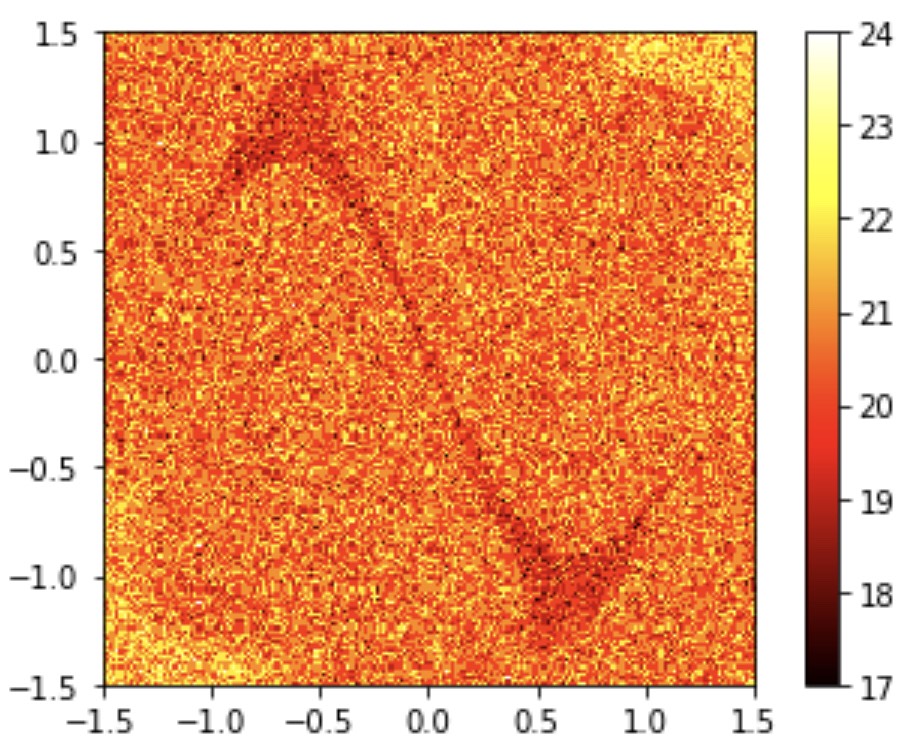}
  \subcaption{$\Gamma_3$}
 \end{minipage}
 \caption{Reconstruction by the horizontal indicator function $I_{hor}$}
 \vspace{5mm}
\end{figure}
%%%%%%%%%%%%%%%%%%%%%%%%%%%%%%%%%%%%%%%%%%%%%%%%%%%%%%%%%%%%%%%%%%%%%%%%%%%%%%%%%%%%%%%%%%%%%%%%%%%%%%%%%%%%%%%%%%%%%%%%%%%%%%%%%%%%%%%%%%%%%%%%%%%%%%%%%%%%%%%%%%%%%%%%%%%%%%%%%%%%%%%%%%%%%%%%%%%%%%%%%%%%%%%%%%%%%%%%%%%%%%%%%%%%%%%%%%%%%%%%%%%%%%%
%%%%%%%%%%%%%%%%%%%%%%%%%%%%%%%%%%%%%%%%%%%%%%%%%%%%%%%%%%%%%%%%%%%%%%%%%%%%%%%%%%%%%%%%%%%%%%%%%%%%%%%%%%%%%%%%%%%%%%%%%%%%
\section*{Acknowledgments}
Authors thank to Professor Bastian von Harrach, who gave us helpful comments in our study.

%%%%%%%%%%%%%%%%%%%%%%%%%%%%%%%%%%%%%%%%%%%%%%%%%%%%%%%%

\noindent E-mail address: takashi.furuya0101@gmail.com


\begin{thebibliography}{17}
\bibitem{D. Colton and R. Kress}D. Colton, R. Kress, {\em Inverse acoustic and electromagnetic scattering theory}, 
Third edition. Applied Mathematical Sciences, {\bf 93} Springer, New York, (2013).

\bibitem{R. Griesmaier and B. Harrach}R. Griesmaier, B. Harrach, {\em Monotonicity in inverse medium scattering on unbounded domains}, SIAM J. Appl. Math. {\bf78}, (2018), no. 5, 2533--2557.

\bibitem{B. Harrach and V. Pohjola and M. Salo1}
B. Harrach, V. Pohjola, M. Salo, {\em Dimension bounds in monotonicity methods for the Helmholtz equation}, SIAM J. Appl. Math. {\bf51}, (2019), no. 4, 2995--3019.

\bibitem{B. Harrach and V. Pohjola and M. Salo2}
B. Harrach, V. Pohjola, M. Salo, {\em  Monotonicity and local uniqueness for the Helmholtz equation}, Analysis and PDE, {\bf12}, (2019), no. 7, 1741--1771.

\bibitem{B. Harrach and M. Ullrich1}
B. Harrach, M. Ullrich, {\em  Local uniqueness for an inverse boundary value problem with partial data}, Proc. Amer. Math. Soc. {\bf 145}, (2017), no. 3, 1087--1095.


\bibitem{B. Harrach and M. Ullrich2}
B. Harrach, M. Ullrich, {\em Monotonicity based shape reconstruction in electrical impedance tomography}, SIAM J. Math. Anal. {\bf 45}, (2013), no. 6, 3382--3403.

\bibitem{Kirsch and Grinberg}A. Kirsch and N. Grinberg, {\em The factorization method for inverse problems}, Oxford University Press, (2008).

\bibitem{Kirsch and Ritter}A. Kirsch and S. Ritter, {\em A linear sampling method for inverse scattering from an open arc}, Inverse Problems {\bf 16}, (2000), no. 1, 89--105.

\bibitem{Kress}R. Kress, {\em Inverse scattering from an open arc}, Math. Methods Appl. Sci., {\bf 18}, (1995), no. 4, 267--293.

\bibitem{Lakshtanov and Lechleiter}
E. Lakshtanov, A. Lechleiter, {\em Difference factorizations and monotonicity in inverse medium scattering for contrasts with fixed sign on the boundary}, SIAM J. Math. Anal. {\em 48} (2016), no. 6, 3688--3707.

\bibitem{McLean}W. McLean,
{\em Strongly elliptic systems and boundary integral equations}, Cambridge University Press, Cambridge, (2000).

\end{thebibliography}
\end{document}